\documentclass[a4paper]{article}

\usepackage[scale=0.7]{geometry}
\usepackage{graphicx}
\usepackage{color}
\usepackage{amsmath,amsfonts,amsthm}
\usepackage{url}
\usepackage{multirow}
\usepackage{hyperref}

\usepackage{pgfplots}

\newcommand{\bb}{\boldsymbol}
\newcommand{\CC}{\mathbb C}
\newcommand{\RR}{\mathbb R}
\newcommand{\ee}{\mathrm e}
\newtheorem{theorem}{Theorem}
\newtheorem{remark}{Remark}

\newenvironment{keywords}{\noindent\textbf{Keywords:}}{}

\title{Direction splitting of $\varphi$-functions in exponential integrators
  for $d$-dimensional problems in Kronecker form}
\author{Marco Caliari\footnote{Department of Computer Science,
    University of Verona, Italy, \texttt{marco.caliari@univr.it},
  corresponding author}
  \and Fabio Cassini\footnote{Department of Computer Science,
    University of Verona, Italy, \texttt{fabio.cassini@univr.it}}}

\begin{document}
  \maketitle
  \begin{abstract}
  In this manuscript, we propose an efficient, practical and easy-to-implement
  way to approximate actions of $\varphi$-functions for matrices with
  $d$-dimensional
  Kronecker sum structure in the context of exponential integrators up
  to second order. The
  method is based on a direction splitting of the involved matrix functions,
  which lets us exploit the highly efficient level 3 BLAS for the actual
  computation of the required actions in a $\mu$-mode fashion. The approach
  has been successfully tested on two- and three-dimensional
  problems with various exponential
  integrators, resulting in a consistent speedup with respect to
  a technique designed to compute actions of $\varphi$-functions
  for Kronecker sums.
  \end{abstract}
  \section{Introduction}
  The problem of computing actions of exponential and exponential-like
  functions
  with Kronecker sum structure received a lot of attention in the last
  years~\cite{CCEOZ22,CCZ23,CCZ23kp,CMM23,LZL21,LZZ22,MMPC22}.
  Indeed, the efficient approximation of such quantities allows to effectively
  employ
  exponential integrators for the time integration of large stiff systems of
  Ordinary
  Differential Equations (ODEs).
  More in detail, we suppose to work with the following stiff system of ODEs
\begin{subequations}\label{eq:ODE}
  \begin{equation}\label{eq:ODEeq}
    \left\{
    \begin{aligned}
      \boldsymbol u'(t)&=K\boldsymbol u(t)+\boldsymbol g(t,\boldsymbol u(t)),
      \quad t>0,\\
      \boldsymbol u(0)&=\boldsymbol u_0.
    \end{aligned}\right.
  \end{equation}
  Here $\boldsymbol g(t,\boldsymbol u(t))$ is a generic nonlinear function
  of $t$
  and of the unknown $\boldsymbol u(t) \in \CC^N$, with $N=n_1\cdots n_d$,
  while $K\in\CC^{N\times N}$ is a matrix with $d$-dimensional
  Kronecker sum structure, i.e.,
  \begin{equation}\label{eq:kronsum}
    K = A_d \oplus A_{d-1} \oplus \cdots \oplus A_1 =
    \sum_{\mu=1}^d A_{\otimes\mu},
    \quad A_{\otimes\mu} = I_d \otimes \cdots \otimes I_{\mu+1} \otimes A_\mu
    \otimes I_{\mu-1} \otimes \cdots \otimes I_1,
  \end{equation}
  \end{subequations}
  where $A_\mu\in \CC^{n_\mu\times n_\mu}$,  and $I_\mu$ is
  the identity matrix of size $n_\mu$. Here and throughout the paper the
  symbol~$\otimes$ denotes the standard Kronecker product of matrices,
  while~$\oplus$ is employed for the Kronecker sum of matrices. Moreover, we
  refer to system~\eqref{eq:ODE} as a system in \emph{Kronecker form} or with
  \emph{Kronecker sum structure}.

  This kind of systems naturally arises in many contexts. For example, for
  $d=2$, such a structure appears in constant coefficient
  matrix Riccati differential equations
  (see, for instance, Reference~\cite[Ch.~3]{AKFIJ03})
  \begin{equation}\label{eq:Riccati}
    \left\{\begin{aligned}
    \boldsymbol U'(t)&=A_1\boldsymbol U(t)+\boldsymbol U(t)A_2^{\mathsf T}+C
    +\boldsymbol U(t)B\boldsymbol U(t),\\
    \boldsymbol U(0)&=\boldsymbol U_0,
  \end{aligned}\right.
  \end{equation}
  where 
  $\boldsymbol U(t)\in\CC^{n_1\times n_2}$, $B\in\CC^{n_2\times n_1}$,
  and $C\in\CC^{n_1\times n_2}$.
  Indeed, using the properties of the Kronecker product~\cite{CCZ23kp}, we can
  rewrite equivalently
  such a matrix equation as a system of ODEs in Kronecker form~\eqref{eq:ODE},
  i.e.,
  \begin{equation}\label{eq:ODE2d}
    \left\{
    \begin{aligned}
      \boldsymbol u'(t)&=((I_2\otimes A_1)+(A_2\otimes I_1))\boldsymbol u(t)
      +\mathrm{vec}(C+\boldsymbol U(t)B\boldsymbol U(t)),\\
      \boldsymbol u(0)&=\mathrm{vec}(\boldsymbol U_0),
    \end{aligned}\right.
  \end{equation}
  where $\mathrm{vec}$ is the operator which stacks the columns of the input
  matrix in a single vector.

  Systems with Kronecker sum structure often arise also when applying
  the method of lines to approximate numerically the solution of a Partial
  Differential Equation (PDE) defined on a tensor product domain and
  appropriate boundary conditions.
  Indeed, after semi-discretization in space
  of well-known parabolic equations such as
  Allen--Cahn, Brusselator, Gray--Scott,
  advection--diffusion--reaction~\cite{CCZ23,CCZ23kp}
  or Schr\"odinger equations~\cite{CCEOZ22},
  we obtain a large stiff system of ODEs in form~\eqref{eq:ODE}.

  Once system~\eqref{eq:ODE} is given, many techniques can be employed to
  numerically integrate it in time, and in particular we are interested in the
  application of exponential integrators~\cite{HO10}. In fact, they are a
  prominent
  way to perform the required task since they enjoy favorable
  stability properties
  that make them suitable to work in the stiff regime. These kinds of schemes
  require the computation of the action of the matrix exponential and of
  exponential-like matrix
  functions (the so-called $\varphi$-functions) on vectors. They are
  defined, for a
  generic matrix $X\in\CC^{N\times N}$, as
\begin{subequations}
  \begin{equation}\label{eq:integraldef}
    \varphi_0(X) = \ee^X, \quad \varphi_\ell(X) = \frac{1}{(\ell-1)!}\int_0^1
    \ee^{(1-\theta)X}\theta^{\ell-1}d\theta, \quad \ell > 0,
  \end{equation}
  and their Taylor series expansion is given by
  \begin{equation}\label{eq:seriesdef}
    \varphi_\ell(X) = \sum_{i=0}^{\infty} \frac{X^i}{(i+\ell)!},
    \quad \ell \geq 0.
  \end{equation}
  \end{subequations}
  When the size of $X$ allows, it is common in practice to approximate
  such matrix functions by means of diagonal Pad\'e
  approximations~\cite{AMH10,BSW07,SW09}
  or via polynomial approximations~\cite{LYL22}.
  On the other hand, when $X$ is large sized, this approach is computationally
  unfeasible, and many algorithms have been developed to perform directly the
  action of $\varphi$-functions on vectors. We mention, among the others,
  Krylov-based techniques~\cite{GRT18,LPR19,NW12},
  direct polynomial methods~\cite{AMH11,CCZ20,CKOR16,LYL22},
  and hybrid techniques~\cite{CCZ22b}.
When $X$ is in fact a matrix $K$ with Kronecker sum
structure~\eqref{eq:kronsum},
  it is possible to exploit this information to compute more efficiently the
  action of the $\varphi$-functions on a vector. Indeed, let us consider
  $\ell=0$, so that $\varphi_0(K) = \ee^K$. Then, it is easy to
  see~\cite{CCZ23kp} that computing
  \begin{equation}\label{eq:vecexp}
    \bb e = \ee^K\bb v=
    \ee^{A_d\oplus A_{d-1}\oplus\cdots\oplus A_1}\bb v=
    \left(\ee^{A_d}\otimes \ee^{A_{d-1}}\otimes\cdots\otimes\ee^{A_1}\right)
        \bb v
  \end{equation}
  is mathematically equivalent to compute
\begin{equation}\label{eq:tuckerexp}
  \bb E=\bb V \times_1 \ee^{A_1}\times_2
  \cdots\times_d \ee^{A_d},
\end{equation}
which we refer to as the \emph{tensor formulation}.
Here, $\bb E$ and $\bb V$ are order-$d$ tensors of size
$n_1\times\cdots\times n_d$ that satisfy
$\mathrm{vec}(\bb E)=\bb e$ and $\mathrm{vec}(\bb V)=\bb v$, respectively,
while $\mathrm{vec}$ is the operator which
stacks the columns of the input tensor into a suitable single column vector.
The symbol $\times_\mu$ denotes the tensor--matrix product along the
mode $\mu$,
which is also known as $\mu$-mode product, and the computation of consecutive
$\mu$-mode products (as it happens in formula~\eqref{eq:tuckerexp})
is usually referred to as \emph{Tucker operator}.
Notice that the element $e_{i_1\ldots i_d}$ of the tensor $\bb E$ turns
out to be
\begin{equation}\label{eq:forloop}
  e_{i_1\ldots i_d}=\sum_{j_d=1}^{n_d}\cdots\sum_{j_1=1}^{n_1}v_{j_1\ldots j_d}
  \prod_{\mu=1}^d
  e^\mu_{i_\mu j_\mu},\quad 1\le i_\mu\le n_\mu,
\end{equation}
  being $e^{\mu}_{i_\mu j_\mu}$ the generic element of $\ee^{A_\mu}$.
  Although formulas~\eqref{eq:vecexp}, \eqref{eq:tuckerexp}, and
  \eqref{eq:forloop} are mathematically equivalent,
  the direct usage of both formulas~\eqref{eq:vecexp} and
  \eqref{eq:forloop} is much less efficient than formula~\eqref{eq:tuckerexp},
  which is
  implemented by exploiting the highly performant level 3 BLAS after computing
  the \emph{small} sized matrix exponentials $\ee^{A_\mu}$.
  Indeed, for instance, formula~\eqref{eq:tuckerexp} in two dimensions requires
  two matrix-matrix products, as it reduces to
  $\ee^{A_1}\bb V\left(\ee^{A_2}\right)^\mathsf{T}$, while
  in the $d$-dimensional case it requires $d$ level 3 BLAS calls.
  This technique led
  to the so-called $\mu$-mode integrator~\cite{CCEOZ22}, and has
  been successfully used to integrate in time semi-discretizations of
  advection--diffusion--reaction
  and Schr\"odinger equations, eventually in
  combination with a splitting scheme. In particular, it is reported
  a consistent speedup
  with respect to state-of-the-art techniques to compute the action
  of the matrix exponential on a vector, as well as a very good scaling when
  performing GPUs simulations.
  We invite a reader interested in more details and applications
  of the Tucker operator to check References~\cite{CCEOZ22,CCZ23kp}.

  When computing actions of $\varphi$-function of higher order, i.e.,
  $\varphi_\ell(K)\bb v$ with $\ell>0$, the last equality in
  formula~\eqref{eq:vecexp} does not hold anymore.
  In Reference~\cite{CCZ23} the authors propose an approach to overcome
  this difficulty,
  by developing a method based on the application of a quadrature formula
  to the integral definition of the $\varphi$-functions~\eqref{eq:integraldef}
  which requires, for each quadrature point,  the action
  of the matrix exponential performed by a Tucker
  operator on the tensor~$\boldsymbol V$.
  In this way, it is
  possible to compute the required action of $\varphi$-functions at a given
  tolerance.
  The technique, which has been named \textsc{phiks},
  has been developed for arbitrary dimension $d$ and
  is designed to compute not only $\varphi$-functions applied to a
  vector but also linear combinations of actions of $\varphi$-functions.
  In addition
  the desired quantities can be made available simultaneously at suitable
  different time scales.
  These features allow to implement high stiff order exponential
  integrators, such as exponential Runge--Kutta schemes, in a more
  efficient way compared to the usage of state-of-the-art techniques
  to compute combinations of actions of $\varphi$-functions.
  Another very recent method based on quadrature rules applied
  to formula~\eqref{eq:integraldef} is presented in
  Reference~\cite{CMM23}, where the technique is described only in dimension
  $d\le 3$ for actions of single $\varphi$-functions at a given time scale.
  Other approaches for the action of $\varphi$-functions for matrices with
  Kronecker sum structure are available in the literature. We mention for
  instance
  Reference~\cite{MMPC22}, whose algorithm is based on the solution of Sylvester
  equations and is currently limited to dimension $d=2$.
  Another way to approximate
  the action of $\varphi$-functions of the Sylvester operator
  $A_1\boldsymbol V+\boldsymbol VA_2^{\sf T}$ or the Lyapunov operator
  $A\boldsymbol V+\boldsymbol VA^{\sf T}$ for the solution of Riccati
  differential equations, possibly in the context of low-rank approximation,
  is presented in References~\cite{LZL21,LZZ22}.

  In this manuscript we propose an alternative way to approximate
  $\varphi_\ell(K)\bb v$, with $\ell>0$ and $K$ a matrix with $d$-dimensional
  Kronecker sum structure, in the context of exponential integrators up to
  second order.
  The approach, that we call \textsc{phisplit}, is based on a direction
  splitting of the matrix $\varphi$-functions of $K$, which
  generates an approximation error compatible with the one of the time marching
  numerical scheme. The evaluation of the required actions is performed
  in a $\mu$-mode fashion by means of a \emph{single} Tucker operator for each
  $\varphi$-function, exploiting the highly efficient level 3 BLAS.
  After recalling some popular exponential integrators in
  Section~\ref{sec:expint},
  we describe in Section~\ref{sec:dirsplit} the proposed technique, as well
  as how to employ it to implement the just mentioned exponential schemes.
  Then, in Section~\ref{sec:numexp}
  we present some numerical experiments that show the effectiveness of
  \textsc{phisplit}, and we finally draw some conclusions in
  Section~\ref{sec:conclusions}.

  \section{Recall of some exponential integrators up to order
    two}\label{sec:expint}
When numerically integrating stiff semilinear ODEs in form~\eqref{eq:ODE},
where the stiff part is represented by the matrix $K$, a prominent approach
is to use exponential integrators~\cite{HO10}. For
convenience of the reader, we report here
(for simplicity in a constant time step size scenario)
a possible derivation of the exponential schemes
that will be employed
later in the numerical experiments of Section~\ref{sec:numexp}.

The starting point is the variation-of-constants formula
\begin{equation}\label{eq:voc}
\begin{split}
  \boldsymbol u(t_{n+1})&=\ee^{\tau K}\boldsymbol u(t_n)+\int_{t_n}^{t_{n+1}}
  \ee^{(t_{n+1}-s)K}\boldsymbol g(s,\boldsymbol u(s))ds\\
  &=\ee^{\tau K}\boldsymbol u(t_n)+\tau\int_0^1
  \ee^{(1-\theta)\tau K}
  \boldsymbol g(t_n+\tau\theta,\boldsymbol u(t_n+\tau\theta))d\theta
\end{split}
\end{equation}
which expresses the analytical solution of system~\eqref{eq:ODEeq} at time
$t_{n+1}=t_n+\tau$, where $\tau$ is the time step size.
If we approximate
the integral with the rectangle left rule, we get the scheme
\begin{equation}\label{eq:lawsoneuler}
  \boldsymbol u_{n+1}=
\ee^{\tau K}\boldsymbol u_n+\tau
  \ee^{\tau K}\boldsymbol g(t_n,\boldsymbol u_n)=
  \ee^{\tau K}(\boldsymbol u_n+\tau
  \boldsymbol g(t_n,\boldsymbol u_n)),
\end{equation}
which is known as Lawson--Euler scheme
(see Reference~\cite[Sec.~A.1.1]{BSW05}).
It is of order one and exact for linear homogeneous problems
with constant coefficients.
The linear part of system~\eqref{eq:ODEeq}
is solved exactly and thus no restriction on the time step due to the stiffness
is necessary.
If instead the trapezoidal quadrature rule is applied to the integral in
equation~\eqref{eq:voc}, we get the approximation
\begin{equation*}
  \boldsymbol u(t_{n+1})\approx\ee^{\tau K}\boldsymbol u(t_n)+
  \frac{\tau}{2}\left(\ee^{\tau K}\boldsymbol g(t_n,\boldsymbol u(t_n))+
  \boldsymbol g(t_{n+1},\boldsymbol u(t_{n+1}))\right).
\end{equation*}
An explicit time marching scheme is then obtained by creating an intermediate
stage $\boldsymbol u_{n2}$
which approximates $\boldsymbol u(t_{n+1})$ in the right hand side
by the Lawson--Euler scheme~\eqref{eq:lawsoneuler}. Overall, we get
\begin{equation}\label{eq:lawson2}
  \begin{aligned}
    \boldsymbol u_{n2}&=\ee^{\tau K}(\boldsymbol u_n+
    \tau\boldsymbol g(t_n,\boldsymbol u_n)),\\
    \boldsymbol u_{n+1}&=\ee^{\tau K}\left(\boldsymbol u_n+
    \frac{\tau}{2}
    \boldsymbol g(t_n,\boldsymbol u_{n})\right)+
    \frac{\tau}{2}\boldsymbol g(t_{n+1},\boldsymbol u_{n2}),
  \end{aligned}
\end{equation}
which is an Lawson method of order two, also known in literature
as Lawson2b (see Reference~\cite[Sec.~A.1.6]{BSW05}).

A different approach to the approximation of the integral in
formula~\eqref{eq:voc} leads to the so-called exponential Runge--Kutta methods.
Indeed, if we approximate only the nonlinear function
$\bb g(t_n+\tau\theta,\bb u(t_n+\tau\theta))$
by
$\boldsymbol g(t_n,\boldsymbol u(t_n))$, by using the definition
of $\varphi_1$ function in equation~\eqref{eq:integraldef} we get the
scheme
\begin{equation*}
  \boldsymbol u_{n+1}=\ee^{\tau K}\boldsymbol u_n+
  \tau\varphi_1(\tau K)\boldsymbol g(t_n,\boldsymbol u_n),
\end{equation*}
which can be equivalently rewritten as
\begin{equation}\label{eq:expEuler}
  \boldsymbol u_{n+1}=\boldsymbol u_n+
  \tau\varphi_1(\tau K)(K\boldsymbol u_n+\boldsymbol g(t_n,\boldsymbol u_n))
\end{equation}
and is known as exponential Euler (or exponential N\o{}rsett--Euler,
see Reference~\cite[Sec.~A.2.1]{BSW05}).
It is a first order scheme and it is exact for linear problems with constant
coefficients.
Another possibility is to interpolate
$\boldsymbol g(t_n+\tau\theta,\boldsymbol u(t_n+\tau\theta))$ with a polynomial
of degree one in $\theta$ at~$0$ and~$1$, thus
obtaining the approximation
\begin{equation*}
  \boldsymbol u(t_{n+1})\approx
  \ee^{\tau K}\boldsymbol u(t_n)+\tau\int_0^1
  \ee^{(1-\theta)\tau K}(\theta\boldsymbol g(t_{n+1},\boldsymbol u(t_{n+1}))+
  (1-\theta)\boldsymbol g(t_n,\boldsymbol u(t_n)))d\theta.
\end{equation*}
By taking a stage $\boldsymbol u_{n2}$
which approximates $\boldsymbol u(t_{n+1})$ in the
right hand side by the exponential Euler scheme and
using the definitions of $\varphi_1$ and $\varphi_2$ functions in
formula~\eqref{eq:integraldef}, we obtain
the second order exponential Runge--Kutta
scheme (also known in literature as ETD2RK,
see Reference~\cite[Sec.~A.2.5]{BSW05})
\begin{equation}\label{eq:ETD2RK}
  \begin{aligned}
    \boldsymbol u_{n2}&=\boldsymbol u_n+
    \tau\varphi_1(\tau K)(K\boldsymbol u_n+
    \boldsymbol g(t_n,\boldsymbol u_n)),\\
    \boldsymbol u_{n+1}&=\boldsymbol u_{n2}+
    \tau\varphi_2(\tau K)(\boldsymbol g(t_{n+1},\boldsymbol u_{n2})-
    \boldsymbol g(t_n,\boldsymbol u_n)).
  \end{aligned}
\end{equation}

Finally, we consider the Rosenbrock--Euler method
(see Reference~\cite[Ex.~2.20]{HO10}) which, in the autonomous case,
can be obtained from
the application of the exponential Euler scheme to the linearized
differential equation
\begin{equation*}
  \boldsymbol u'(t)=\left(K+
  \frac{\partial \boldsymbol g}{\partial \boldsymbol u}(\boldsymbol u_n)\right)
  \boldsymbol u(t)
  +\left(\boldsymbol g(\boldsymbol u(t))-
  \frac{\partial \boldsymbol g}{\partial \boldsymbol u}(\boldsymbol u_n)
  \boldsymbol u(t)\right)
  =K_n\boldsymbol u(t)+\boldsymbol g_n(\boldsymbol u(t)),
\end{equation*}
where $K_n$ is the Jacobian evaluated at $\boldsymbol u_n$
and $\boldsymbol g_n(\boldsymbol u(t))$ the
remainder.
The resulting scheme is
\begin{equation}\label{eq:expRosEuler}
  \boldsymbol u_{n+1}=
  \boldsymbol u_n+
  \tau\varphi_1(\tau K_n)(K\boldsymbol u_n+\boldsymbol g(\boldsymbol u_n)).
\end{equation}
It is a second order method and, in contrast to all the methods
presented above, it requires the evaluation of
a different matrix function
$\varphi_1(\tau K_n)$ at \emph{each time step}.
The extension to non-autonomous
systems is straightforward,
see Reference~\cite[Ex.~2.21]{HO10}.
\begin{remark}
  We considered here only a selected number of exponential integrators
  which require the action
  of $\varphi$-functions. Other exponential-type schemes of
  first or second order could benefit
  from the $\mu$-mode splitting technique for computing
  $\varphi$-functions of Kronecker sums that we present in this work.
  We mention, among the others,
  corrected splitting schemes~\cite{EO15}, low-regularity schemes~\cite{RS21},
  and Magnus integrators for linear time dependent coefficient
  non-homogeneous equations~\cite{GOT06}.
\end{remark}

\section[Direction splitting of phi-functions]%
        {Direction splitting of $\varphi$-functions}\label{sec:dirsplit}
  As mentioned in the introduction, we suppose that we are dealing with a matrix
  $K$ with Kronecker sum structure~\eqref{eq:kronsum},
  and we are interested in computing efficiently
  $\varphi_\ell(\tau K)\boldsymbol v$, $\boldsymbol v \in \CC^N$,
  $\tau \in \RR$, in the
  context of exponential integrators.
  In particular, we know that by employing a scheme of order~$p$,
  we make a local error $\mathcal{O}(\tau^{p+1})$, being $\tau$ the (constant)
  time step size. Hence, if the integrator requires to compute a quantity
  of the form $\tau^q\varphi_\ell(\tau K)$, with $q>0$, it is sufficient to
  approximate $\varphi_\ell(\tau K)$ with an error $\mathcal{O}(\tau^{p+1-q})$,
  to preserve the order of convergence. For our schemes of interest,
  i.e., the ones presented in the previous section, we make use of the following
  result.
  \begin{theorem}
    Let $K$ be a matrix with $d$-dimensional Kronecker sum
    structure~\eqref{eq:kronsum}. Then, for $\ell>0$, we have
  \begin{equation}\label{eq:dsphi}
    \varphi_\ell(\tau K) = (\ell!)^{d-1}\left(
    \varphi_\ell(\tau A_d)\otimes\varphi_\ell(\tau A_{d-1})\otimes
    \cdots\otimes\varphi_\ell(\tau A_1))
    \right)
    +\mathcal{O}(\tau^2).
    \end{equation}
\end{theorem}
\begin{proof}
For compactness of presentation, we employ the following notation
\begin{equation*}
  X_d \otimes X_{d-1} \otimes \cdots \otimes X_1 = \bigotimes_{\mu=d}^1 X_\mu,
  \quad X_\mu\in\CC^{n_\mu \times n_\mu}.
\end{equation*}
Then, by using the Taylor expansion of the $\varphi_\ell$
function~\eqref{eq:seriesdef}
and the properties of the Kronecker product (see Reference~\cite{L00}
for a comprehensive
review) we obtain
\begin{equation*}
  \begin{aligned}
    (\ell!)^{d-1}\bigotimes_{\mu=d}^1\varphi_\ell(\tau A_\mu)
    &=(\ell!)^{d-1}\bigotimes_{\mu=d}^1\left(\frac{I_\mu}{\ell!}+
  \frac{\tau A_\mu}{(\ell+1)!}+\mathcal{O}(\tau^2)\right)\\
  &=(\ell!)^{d-1}\left(\frac{1}{(\ell!)^{d}}\bigotimes_{\mu=d}^1 I_\mu
  +\frac{\tau}{(\ell!)^{d-1}(\ell+1)!}\sum_{\mu=1}^dA_{\otimes\mu}+
  \mathcal{O}(\tau^2)
  \right) \\
  &= \frac{I}{\ell!}+\frac{\tau K}{(\ell+1)!}+\mathcal{O}(\tau^2)\\
  &= \varphi_\ell(\tau K) + \mathcal{O}(\tau^2),
  \end{aligned}
\end{equation*}
where $I$ is the identity matrix of size $N\times N$.
\end{proof}%
Formula~\eqref{eq:dsphi} allows for an efficient $\mu$-mode based
implementation,
similarly to the matrix exponential case~\eqref{eq:tuckerexp}. Indeed, given
an order-$d$ tensor $\boldsymbol V$ such that
$\boldsymbol v = \mathrm{vec}(\boldsymbol V)$, if we define
\begin{equation*}
  \boldsymbol p_\ell^{(2)}=\varphi_{\ell}^{(2)}(\tau K)\boldsymbol v=
  \mathrm{vec}\left(((\ell!)^{d-1}\boldsymbol V)\times_1
  \varphi_{\ell}(\tau A_1)\times_2
  \varphi_{\ell}(\tau A_2)\times_3 \cdots\times_d   \varphi_{\ell}(\tau A_d)
  \right)
\end{equation*}
we have
\begin{equation*}
  \varphi_{\ell}(\tau K)\boldsymbol v =
  \boldsymbol p_\ell^{(2)}+ \mathcal{O}(\tau^2).
\end{equation*}
We refer to
\begin{equation*}
  \boldsymbol P_\ell^{(2)}=\left((\ell!)^{d-1}\boldsymbol V\right)\times_1
  \varphi_{\ell}(\tau A_1)\times_2
  \varphi_{\ell}(\tau A_2)\times_3 \cdots\times_d   \varphi_{\ell}(\tau A_d)
\end{equation*}
as the \emph{tensor formulation} of $\boldsymbol p_\ell^{(2)}$,
so that $\mathrm{vec}(\boldsymbol P_\ell^{(2)})=\boldsymbol p_\ell^{(2)}$.
This is precisely the formulation that we propose to employ when
actions of $\varphi$-functions of a matrix with Kronecker sum structure
are required for
the above exponential integrators.
From now on, we refer to this technique as the \textsc{phisplit} approach.
Notice that, after the computation of the \emph{small} sized matrix functions
$\varphi_\ell(\tau A_\mu)$, with $\mu=1,\ldots,d$, a single Tucker operator is
required to evaluate the tensor formulation approximation.

\subsection[Evaluation of small sized matrix phi-functions]%
           {Evaluation of small sized matrix
             $\varphi$-functions}\label{sec:phiquad}
The matrices $A_\mu$ have a much smaller size compared to $K$, and
the corresponding matrix $\varphi$-functions can be directly computed without
much effort. In particular, for $\varphi_0$ (i.e., the exponential function),
we employ the most popular technique for generic matrices, which is based
on a diagonal
rational Pad\'e approximation coupled with a scaling and squaring
algorithm (see Reference~\cite{AMH10}). This procedure is encoded in the
internal
\textsc{matlab} function \texttt{expm}.
Different algorithms that could be used as well, based on Taylor approximations
of the matrix exponential, can be found in References~\cite{CZ19,SID19}.

For the computation of higher order matrix $\varphi$-functions we rely on a
quadrature formula applied to the integral definition~\eqref{eq:integraldef}.
For a generic matrix $X\in\CC^{N\times N}$, we have then
\begin{equation}\label{eq:phiquad}
  \varphi_\ell(X) \approx \sum_{i=1}^q w_i\ee^{(1-\theta_i)X}
  \frac{\theta_i^{\ell-1}}{(\ell-1)!}, \quad \ell>0.
\end{equation}
In order to avoid an impractically large number of quadrature points, we couple
the procedure with a modified scaling and squaring algorithm (see
Reference~\cite{SW09}). In fact, we scale the original matrix $X$ by $2^s$,
where $s$ is a natural number defined so that $\lVert X/2^s\rVert_1<1$,
we approximate $\varphi_\ell(X/2^s)$ by means
of formula~\eqref{eq:phiquad} and we recover $\varphi_\ell(X)$ by the recurrence
\begin{equation}\label{eq:squarephi}
  \varphi_{\ell}(2z)=\frac{1}{2^\ell}\left[\ee^z\varphi_\ell(z)+
    \sum_{k=1}^\ell\frac{\varphi_k(z)}{(\ell-k)!}\right].
\end{equation}
In order to compute the needed matrix exponentials, we again employ the internal
\textsc{matlab} function \texttt{expm}.
Notice that the squaring of $\varphi_\ell$ also requires the evaluation
of all the $\varphi_j$ functions, for $0< j<\ell$. In particular, for the
first squaring step, we compute themselves by
formula~\eqref{eq:phiquad} with $\ell=j$ and using the \emph{same} set of
matrix exponentials
already available for the quadrature procedure of $\varphi_\ell(X/2^s)$.
For all the subsequent squaring steps, we use formula~\eqref{eq:squarephi}
itself
with $\ell=j$.
As a consequence of this procedure, the computation of a single matrix
function $\varphi_\ell$
makes available all the matrix functions $\varphi_j$, with $0\le j\le \ell$.
In practice, for the quadrature we employ the Gauss--Legendre--Lobatto
formula, which
allows for high precision with a moderate number of quadrature nodes. Moreover,
since it uses the endpoints of the quadrature interval, we make use of
the matrix
exponential $\ee^{X/2^s}$ ($\theta_1 = 0$), which is also required
for the subsequent squaring procedure, and we avoid generating the last matrix
exponential since $\theta_q = 1$.
The overall procedure is implemented in \textsc{matlab} language in our
function \texttt{phiquad}.

Alternatively, for the computation of the matrix $\varphi$-functions, it is
possible to employ the \textsc{matlab} routine \texttt{phipade}
(see Reference~\cite{BSW07}),
whose algorithm is based on a rational Pad\'e approximation coupled with
the squaring formula~\eqref{eq:squarephi}.
Another recent technique that employs a polynomial Taylor
approximation instead of rational Pad\'e one is presented in
Reference~\cite{LYL22}.

\subsection{Practical implementation of the exponential
  integrators}\label{sec:methods}
The implementation of the Lawson methods introduced in
Section~\ref{sec:expint}, which require just actions of matrix exponentials,
does not suffer from any direction splitting error, thanks
to the equivalence between formulas~\eqref{eq:vecexp} and~\eqref{eq:tuckerexp}.
In particular, the tensor formulation of Lawson--Euler is
\begin{equation}\label{eq:lawsoneulertensor}
  \bb U_{n+1}=(\bb U_n+\tau \bb G(t_n,\bb U_n))\times_1 \ee^{\tau A_1}\times_2
  \cdots\times_d\ee^{\tau A_d},
\end{equation}
while the Lawson2b scheme is given by
\begin{equation}\label{eq:lawson2tensor}
  \begin{aligned}
    \boldsymbol U_{n2}&=(\bb U_n+\tau \bb G(t_n,\bb U_n))\times_1
    \ee^{\tau A_1}\times_2
  \cdots\times_d\ee^{\tau A_d},\\
  \boldsymbol U_{n+1}&=
  \left(\bb U_n+\frac{\tau}{2} \bb G(t_n,\bb U_n)\right)
  \times_1 \ee^{\tau A_1}\times_2
  \cdots\times_d\ee^{\tau A_d}+\frac{\tau}{2}\bb G(t_{n+1},\bb U_{n2}).
  \end{aligned}
\end{equation}
Here and in the subsequent formulas we have
\begin{equation*}
  \mathrm{vec}(\bb U) = \bb u, \quad
  \mathrm{vec}(\bb G(t, \bb U)) = \bb g(t,\bb u).
\end{equation*}
The remaining exponential integrators transform
as follows.
First of all, the action of the matrix $K$ on $\bb u_n$
is computed in tensor form as
\begin{equation*}
\sum_{\mu=1}^d\left(\bb U_n\times_\mu A_\mu\right)
\end{equation*}
without explicitly assembling the matrix $K$ (see Reference~\cite{CCZ22b}).
The exponential Euler \textsc{phisplit} method is
\begin{equation}\label{eq:expEulertensor}
  \bb U_{n+1}=\bb U_n+\tau
  \left(\sum_{\mu=1}^d(\boldsymbol U_n\times_\mu A_\mu)+\bb G(t_n,\bb U_n)\right)
  \times_1\varphi_1(\tau A_1)\times_2\cdots\times_d\varphi_1(\tau A_d).
\end{equation}
Notice that an alternative version, more appropriate for the cases in which
the contribute of $\bb G(t_n,\bb U_n)$ is particularly small, is
\begin{equation*}
  \bb U_{n+1}=\bb U_n\times_1\ee^{\tau A_1}\times_2\cdots\times_d\ee^{\tau A_d}+
  \tau\bb G(t_n,\bb U_n)
  \times_1\varphi_1(\tau A_1)\times_2\cdots\times_d\varphi_1(\tau A_d),
\end{equation*}
which is in fact an exact formula for linear autonomous problems.
By considering formulation~\eqref{eq:expEulertensor} for the exponential
Euler methods, the ETD2RK \textsc{phisplit} scheme becomes
\begin{equation}\label{eq:ETD2RKtensor}
\begin{aligned}
  \bb U_{n2}&=\bb U_n+\left(\sum_{\mu=1}^d(\bb U_n\times_\mu A_\mu)+
  \tau\bb G(t_n,\bb U_n)\right)
  \times_1\varphi_1(\tau A_1)\times_2\cdots\times_d\varphi_1(\tau A_d),\\
  \bb U_{n+1}&=\bb U_{n2}+\tau\left(\bb G(t_{n+1},\bb U_{n2})
  -\bb G(t_n,\bb U_n)\right)
  \times_1\varphi_2(\tau A_1)\times_2\cdots\times_d\varphi_2(\tau A_d).
\end{aligned}
\end{equation}
Finally, concerning the
exponential Rosenbrock--Euler method for
autonomous systems, we assume that the Jacobian $K_n$
can be written as a Kronecker sum, i.e.,
\begin{equation*}
  K_n=K+\frac{\partial \bb g}{\partial \bb u}(\bb u_n)=
  J_d(\bb U_n)\oplus J_{d-1}(\bb U_n)\oplus\cdots\oplus J_1(\bb U_n).
\end{equation*}
Therefore the exponential Rosenbrock--Euler \textsc{phisplit} method
 is
\begin{equation}\label{eq:expRosEulertensor}
  \bb U_{n+1}=\bb U_n+\left(\sum_{\mu=1}^d (\bb U_n\times_\mu A_\mu)+
  \tau\bb G(\bb U_n)\right)
  \times_1\varphi_1(\tau J_1(\bb U_n))\times_2\cdots
  \times_d\varphi_1(\tau J_d(\bb U_n)).
\end{equation}
  \section{Numerical experiments}\label{sec:numexp}
In this section we present numerical experiments to validate
the proposed approach \textsc{phisplit}.
In particular, we will
consider a two-dimensional example from linear quadratic control and a
three-dimensional
example which models an advection--diffusion--reaction equation. To perform the
time marching, we will employ the exponential
integrators of Section~\ref{sec:expint}
as described in Section~\ref{sec:methods} for the \textsc{phisplit} version.

As term of comparison, we will consider the approximation of actions of
$\varphi$-functions for matrices with Kronecker sum structure
using
\textsc{phiks}\footnote{\url{https://github.com/caliarim/phiks}}~\cite{CCZ23}.
This algorithm operates in tensor formulation using $\mu$-mode products, too,
but it requires an input tolerance, which we take proportional
to the local temporal order of the method and to the norm of the current
solution.
The proportionality constant is chosen
so that the error committed by the routine, measured against a reference
or analytical
solution, does not affect the temporal error.

To compute all the relevant tensor operations, i.e., Tucker operators and
$\mu$-mode
products, we use the functions contained in the package
KronPACK\footnote{\url{https://github.com/caliarim/KronPACK}}.
Moreover, to compute the needed matrix $\varphi$-functions,
we employ the internal \textsc{matlab} function \texttt{expm}
(for $\varphi_0$) and
the function \texttt{phiquad} (for $\varphi_\ell$, $\ell>0$), as presented in
Section~\ref{sec:phiquad}.
In terms of hardware, we run all the experiments employing an
Intel\textsuperscript{\textregistered}
Core\textsuperscript{\texttrademark} i7-10750H CPU with six physical cores and
16GB of RAM. As a software, we use MathWorks
MATLAB\textsuperscript{\textregistered} R2022a.
All the codes to reproduce the numerical examples, together with our
implementation of \textsc{phisplit} and the function \texttt{phiquad}, can be
found in a maintained GitHub
repository\footnote{\url{https://github.com/caliarim/phisplit}}.
We finally stress that the functions \texttt{phisplit} and
\texttt{phiquad} are fully compatible with GNU Octave.

\subsection{Linear quadratic control}
We present in this section a classical example from linear quadratic
control (see, for instance, References~\cite{MOPP18,P00}).
We are interested in the minimization over the scalar control $v(t)\in\RR$ of
the functional
\begin{equation*}
  \mathcal{J}(v) = \frac{1}{2}\int_0^T\left(\alpha s(t)^2
  + v(t)^2 \right) dt
\end{equation*}
subject to the constraints
\begin{equation*}
  \begin{aligned}
&    w'(t) = Aw(t) + bv(t), \quad w(0) = w_0, \\
&    s(t) = cw(t).
  \end{aligned}
\end{equation*}
Here $w(t)\in\RR^{n\times 1}$ is a column vector containing the state variables,
$s(t)\in\RR$ represents the scalar output,
$A\in\RR^{n\times n}$ is the system matrix,
$b\in\RR^{n\times 1}$ is the system column vector,
$c\in\RR^{1\times n}$ is a row vector, and $\alpha\in\RR^{+}$ is a
positive scalar.

Then, the solution of the constrained optimization problem is determined by
the optimal
control
\begin{equation*}
  v^*(t) = -b^{\mathsf T}\boldsymbol U(t)w(t),
\end{equation*}
where $\boldsymbol U(t)\in\RR^{n\times n}$
satisfies the symmetric Riccati  differential equation
\begin{equation}\label{eq:exric}
\left\{
\begin{aligned}
  &    \boldsymbol U'(t) = A^{\mathsf T}\boldsymbol U(t) + \boldsymbol U(t)A +
  C+ \boldsymbol U(t)B\boldsymbol U(t) ,\\
&    \boldsymbol U(0) = \boldsymbol Z,
\end{aligned}
\right.
\end{equation}
where $C=\alpha c^{\mathsf T}c$ and $B=-bb^\mathsf T$
(see Reference~\cite[Ch.~4]{AKFIJ03} for a comprehensive introduction
to the subject).

Here $\boldsymbol Z\in\RR^{n\times n}$ is a matrix containing all zeros entries.
Clearly, equation~\eqref{eq:exric} is in form~\eqref{eq:Riccati}, which
in turn can be seen as a problem with two-dimensional Kronecker
sum structure~\eqref{eq:ODE2d}
and integrated efficiently by means
of the techniques described in Section~\ref{sec:dirsplit}.
Notice also that the solution of equation~\eqref{eq:exric} converges to a steady
state determined by the algebraic Riccati equation
\begin{equation}\label{eq:steady}
  A^{\mathsf T}\boldsymbol U(t) +
  \boldsymbol U(t)A +C + \boldsymbol U(t)B\boldsymbol U(t) = 0.
\end{equation}

For our numerical experiment, similarly to what previously done in the
literature~\cite{LZL21,LZZ22,MOPP18,P00},
we take $A\in\RR^{\hat n^2\times \hat n^2}$
as the matrix obtained by the discretization with second order
centered finite differences of the operator
\begin{equation}\label{eq:advdiff2dric}
  \partial_{xx}+\partial_{yy} -10x\partial_x -100y\partial_y
\end{equation}
on the domain $[0,1]^2$ with homogeneous Dirichlet boundary conditions.
Moreover,
the components $b_k$ of the vector $b$ are defined as
\begin{equation*}
b_k=\left\{
\begin{aligned}
    1 &\text{ if } 0.1<x_i\leq 0.3,\\
    0 &\text{ otherwise},
\end{aligned}
\right.\quad
k=i+(j-1)\hat{n},\quad i=1,\ldots,\hat{n},\quad j=1,\ldots,\hat{n},
\end{equation*}
while for the components $c_k$ of the vector $c$ we take
\begin{equation*}
c_k=\left\{
\begin{aligned}
    1 &\text{ if } 0.7<x_i\leq 0.9,\\
    0 &\text{ otherwise},
\end{aligned}
\right.\quad
k=i+(j-1)\hat{n},\quad i=1,\ldots,\hat{n},\quad j=1,\ldots,\hat{n}.
\end{equation*}
Here $x_i$ represents the $i$th (inner) grid point along the $x$ direction.
Finally,
we set $\alpha=100$.

For the temporal integration of equation~\eqref{eq:exric} we use
the exponential Rosenbrock--Euler method, already employed in
References~\cite{LZL21,LZZ22},
and reported in formula~\eqref{eq:expRosEuler}
(see formula~\eqref{eq:expRosEulertensor}
for the \textsc{phisplit} version).
In fact, the Jacobian matrix of system~\eqref{eq:exric}
has the following Kronecker sum structure
\begin{equation*}
  K_n=I\otimes(A^\mathsf{T}+\boldsymbol U_nB)+
  (A+B\boldsymbol U_n)^{\sf T}\otimes I,
\end{equation*}
where $I$ is the identity matrix of size $n\times n$, with $n=\hat{n}^2$.
We remark that the exponential Rosenbrock--Euler \textsc{phisplit} method
requires at \emph{each time step} to evaluate
the matrix function
$\varphi_1(\tau (A^\mathsf{T}+\boldsymbol U_nB))$,
to compute the action $K\bb u_n$ 
and to perform one Tucker operator.
We will employ also the second order exponential Runge--Kutta method
ETD2RK, reported in formula~\eqref{eq:ETD2RK} and
presented in \textsc{phisplit} sense in formula~\eqref{eq:ETD2RKtensor}.
Although each time step of this integrator
requires two Tucker operators plus the action $K\bb u_n$
for the \textsc{phisplit} version, in a constant
time step size implementation
the needed matrix functions
$\varphi_1(\tau A^\mathsf{T})$ and
$\varphi_2(\tau A^\mathsf{T})$
can be computed once and for all at the beginning.

First of all, we verify the implementation of the involved
exponential integrators
for a long term simulation, i.e., until reaching the
steady state. For this experiment, we employ $\hat{n}=20$
inner discretization points for the $x$ and the $y$ variables.
As confirmed by the plot in Figure~\ref{fig:steadystate},
we see that all the methods, both in their \textsc{phiks} and \textsc{phisplit}
implementation, reach around time $0.15$ the solution of
equation~\eqref{eq:steady}, which
is obtained with the \textsc{matlab} function \texttt{icare} from the
Control System Toolbox.
\begin{figure}[htb!]
  \centering
  \input{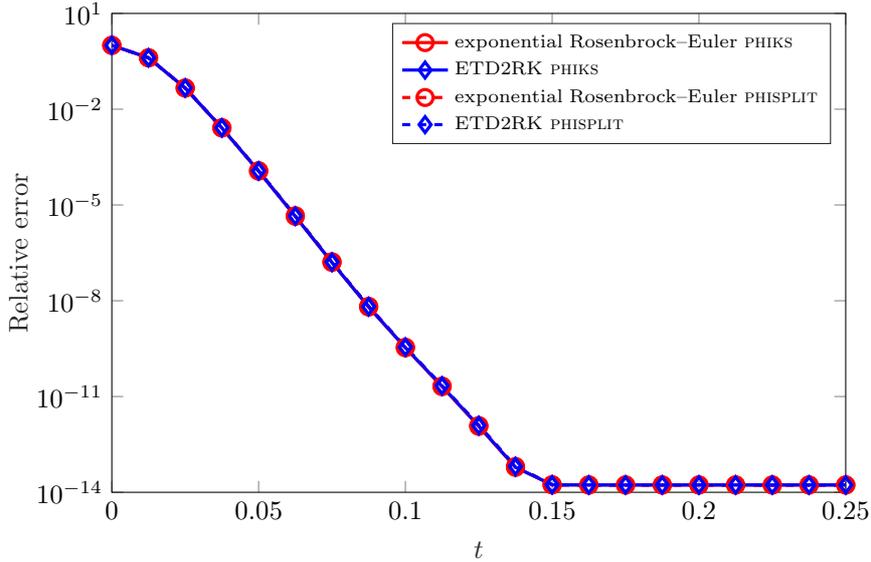}
  \caption{Convergence of the exponential Rosenbrock--Euler and of the ETD2RK
    methods, both in \textsc{phiks} and in \textsc{phisplit} variants,
    to the steady state of Riccati differential equation~\eqref{eq:exric}.
    Here $\hat{n}=20$,
    the integrators have been employed with 200 time steps, and the relative
    errors, measured in the Frobenius norm with respect to the solution of
    algebraic Riccati equation~\eqref{eq:steady}, are displayed each 10th
    time step.}
  \label{fig:steadystate}
\end{figure}

Then, we compare the performances of the integrators for the solution
of equation~\eqref{eq:exric} with $\hat{n}=30$ and final time $T=0.025$.
All methods are run with different time step sizes in such
a way to reach comparable relative errors
with respect to a reference solution. The number of time steps for
each method and simulation, together with the numerically observed
convergence rate, is reported in Table~\ref{tab:riccati_order}.
All the methods appear to be of second order, as expected.
\begin{table}[htb!]
  \centering
  \bgroup
  \def\arraystretch{1.2}
  \begin{tabular}{c||c|ccccc}
  \multirow{2}{*}{exponential Rosenbrock--Euler \textsc{phiks}} &
  steps & 10 & 20 & 30 & 40 & 50 \\\cline{2-7}
  &order &  --  & 2.11 & 2.06 & 2.05 & 2.03 \\
  \hline
  \multirow{2}{*}{ETD2RK \textsc{phiks}} &
  steps & 7 & 14 & 21 & 28 & 35 \\\cline{2-7}
  &order &  --  & 2.08 & 2.05 & 2.03 & 2.03 \\
  \hline
  \multirow{2}{*}{exponential Rosenbrock--Euler \textsc{phisplit}} &
  steps & 30 & 65 & 100 & 135 & 170 \\\cline{2-7}
  &order &  --  & 2.05 & 2.03 & 2.02 & 2.02 \\
  \hline
  \multirow{2}{*}{ETD2RK \textsc{phisplit}} &
  steps & 30 & 65 & 100 & 135 & 170 \\\cline{2-7}
  & order & --   & 2.05 & 2.03 & 2.02 & 2.02
  \end{tabular}
  \egroup
  \caption{Number of time steps and observed convergence rates for the
    time integration of Riccati differential
    equation~\eqref{eq:exric} up to
    final time $T=0.025$, with different exponential integrators and
    $\hat{n}=30$.
    The achieved errors and the wall-clock times are displayed in
    Figure~\ref{fig:riccati}.}
  \label{tab:riccati_order}
\end{table}

In Figure~\ref{fig:riccati} we report the relative errors and the corresponding
wall-clock times of the simulations. Here, we also include the performance
of the built-in \textsc{matlab} function \texttt{ode23}. This is an explicit
Runge--Kutta method of order three with variable step size, suggested for
not stringent tolerances and for moderately stiff problems. In fact,
it turned out to be the fastest routine in the ODE suite to reach
accuracies in the same range of the other methods.
We notice first of all that the exponential Rosenbrock--Euler method
is always faster than ETD2RK in the \textsc{phiks} implementation,
that is with the action of matrix functions computed at a precision that
does not affect the temporal error (see the discussion at the beginning
of the section). On the other hand, the two implementations with
\textsc{phisplit} are always faster compared with
their \textsc{phiks} counterparts, although they require a larger
number of time steps to reach a comparable accuracy. Moreover,
the ETD2RK method turns out to be faster
with respect to the exponential Rosenbrock--Euler method.
This is mainly due to the fact that the matrix functions in the Runge--Kutta
case are computed only once before the time marching. This method is in fact
at least twice as fast as the other exponential methods and faster
than \texttt{ode23}, which anyway shows a good performance
for the most stringent tolerances.
\begin{figure}[htb!]
  \centering
%
%
%
\begin{tikzpicture}

\begin{axis}[%
width=3.8in,
height=2.5in,
at={(0.769in,0.477in)},
scale only axis,
xmin=0,
xmax=45,
xlabel style={font=\color{white!15!black}},
xlabel={wall-clock time},
ymode=log,
ymin=1e-6,
ymax=1e-3,
yminorticks=true,
ylabel style={font=\color{white!15!black}},
ylabel={Relative error},
axis background/.style={fill=white},
legend style={legend cell align=left, align=left, draw=white!15!black,font=\scriptsize},
]

  \addplot [color=red, line width=1.5pt, mark size=3pt,mark=o, mark options={solid, red}]
  table[row sep=crcr]{%
7.245043	0.000164278362395639\\
12.768241	3.80463295039712e-05\\
18.837567	1.64713347918775e-05\\
25.100172	9.14433664075329e-06\\
32.350285	5.80724316399365e-06\\
};
\addlegendentry{exponential Rosenbrock--Euler \textsc{phiks}}

\addplot [color=blue,line width=1.5pt, mark size=3pt, mark=diamond, mark options={solid, blue}]
  table[row sep=crcr]{%
11.096616	0.000125111098081683\\
18.585951	2.95178390772025e-05\\
26.695792	1.287458889008e-05\\
34.362065	7.17497718911103e-06\\
43.492348	4.56600169969307e-06\\
};
\addlegendentry{ETD2RK \textsc{phiks}}

\addplot [color=red,line width=1.5pt, dashed, mark size=3pt, mark=o, mark options={solid, red}]
  table[row sep=crcr]{%
5.643451	0.000161054737647592\\
10.629037	3.31032060605755e-05\\
17.17628	1.38317057433022e-05\\
21.162693	7.54702627959899e-06\\
25.639637	4.74246420866496e-06\\
};
\addlegendentry{exponential Rosenbrock--Euler \textsc{phisplit}}

  \addplot [color=blue, line width=1.5pt, dashed, mark size=3pt,mark=diamond, mark options={solid, blue}]
  table[row sep=crcr]{%
2.630911	0.00017185971101606\\
5.065323	3.53283429241861e-05\\
7.762737	1.47616662339618e-05\\
10.480314	8.05456774635061e-06\\
13.051869	5.06151122382156e-06\\
};
\addlegendentry{ETD2RK \textsc{phisplit}}

\addplot [color=magenta, line width=1.5pt, dashed, mark size=3pt,mark=asterisk, mark options={solid, magenta}]
  table[row sep=crcr]{%
1.55e1	1.84e-4\\
1.56e1	3.36e-05\\
1.61e1	1.66e-05\\
1.63e1	9.58e-06\\
1.64e1	6.36e-06\\
};
\addlegendentry{\texttt{ode23}}

\end{axis}
\end{tikzpicture}%
  \caption{Achieved errors in the Frobenius norm
    and wall-clock times in seconds for the solution
    of Riccati differential equation~\eqref{eq:exric} up to final
    time $T=0.025$, with different
    integrators and $\hat{n}=30$. The number of time steps for each
    exponential method is reported in
    Table~\ref{tab:riccati_order}. The input tolerances (both absolute
    and relative) for \texttt{ode23} are $5\mathrm{e}{-3}$,
    $1\mathrm{e}{-3}$, $1\mathrm{e}{-4}$,
    $5.5\mathrm{e}{-5}$, and $5\mathrm{e}{-5}$.}
  \label{fig:riccati}
\end{figure}
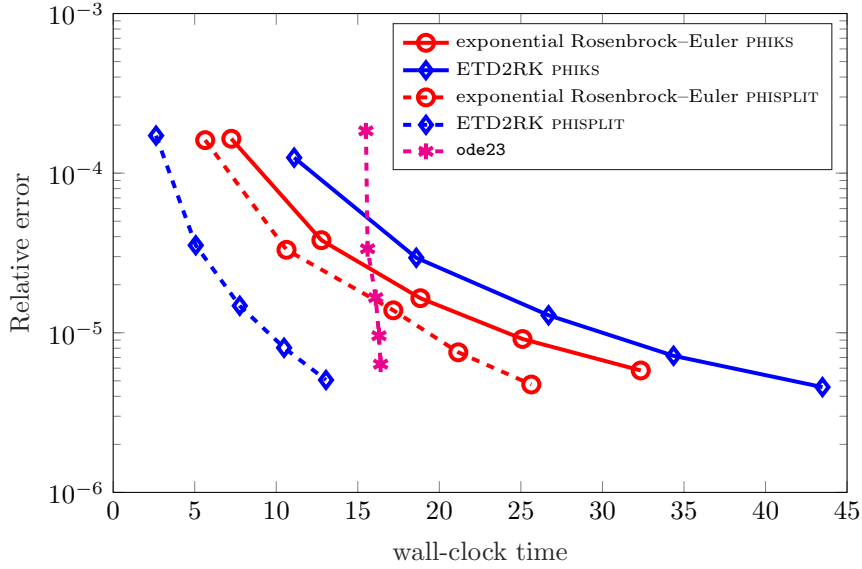

Finally, we repeat the same experiment with $\hat{n}=40$. The results are
presented in Table~\ref{tab:riccati_order_40} and in
Figure~\ref{fig:riccati_40}. The global behavior is similar with respect
to the previous case, although the speed-ups of the \textsc{phisplit}
implementations with respect to their \textsc{phiks} counterparts
is noticeably larger. In fact, ETD2RK \textsc{phisplit} is still
the best performant method.
\begin{table}[htb!]
  \centering
  \bgroup
  \def\arraystretch{1.2}
  \begin{tabular}{c||c|ccccc}
  \multirow{2}{*}{exponential Rosenbrock--Euler \textsc{phiks}} &
  steps & 15 & 30 & 45 & 60 & 75 \\\cline{2-7}
  &order &  --  & 2.10 & 2.06 & 2.04 & 2.03 \\
  \hline
  \multirow{2}{*}{ETD2RK \textsc{phiks}} &
  steps & 10 & 20 & 30 & 40 & 50 \\\cline{2-7}
  &order &  --  & 2.12 & 2.07 & 2.05 & 2.04 \\
  \hline
  \multirow{2}{*}{exponential Rosenbrock--Euler \textsc{phisplit}} &
  steps & 30 & 65 & 100 & 135 & 170 \\\cline{2-7}
  &order &  --  & 2.05 & 2.03 & 2.02 & 2.02 \\
  \hline
  \multirow{2}{*}{ETD2RK \textsc{phisplit}} &
  steps & 30 & 65 & 100 & 135 & 170 \\\cline{2-7}
  & order & --   & 2.05 & 2.03 & 2.02 & 2.02
  \end{tabular}
  \egroup
  \caption{Number of time steps and observed convergence rates for the
    time integration of Riccati differential
    equation~\eqref{eq:exric} up to
    final time $T=0.025$, with different exponential integrators and 
    $\hat{n}=40$.
    The achieved errors and the wall-clock times are displayed in
    Figure~\ref{fig:riccati_40}.}
  \label{tab:riccati_order_40}
\end{table}
\begin{figure}[htb!]
  \centering
%
%
%
\begin{tikzpicture}

\begin{axis}[%
width=3.8in,
height=2.5in,
at={(0.769in,0.477in)},
scale only axis,
xmin=0,
xmax=550,
xlabel style={font=\color{white!15!black}},
xlabel={wall-clock time},
ymode=log,
ymin=1e-06,
ymax=1e-3,
yminorticks=true,
ylabel style={font=\color{white!15!black}},
ylabel={Relative error},
axis background/.style={fill=white},
legend style={legend cell align=left, align=left, draw=white!15!black,font=\scriptsize}
]

\addplot [color=red, line width=1.5pt, mark=o, mark size=3pt, mark options={solid, red}]
  table[row sep=crcr]{%
77.379622	0.000152581963837212\\
150.17592	3.55689099963272e-05\\
223.494456	1.54419043987331e-05\\
297.738835	8.58598664996109e-06\\
371.476896	5.45802490906827e-06\\
};
\addlegendentry{exponential Rosenbrock--Euler \textsc{phiks}}

  \addplot [color=blue, line width=1.5pt, mark size=3pt, mark=diamond, mark options={solid, blue}]
  table[row sep=crcr]{%
108.096225	0.000208297348210989\\
202.231764	4.78052929672065e-05\\
303.59765	2.06260819439754e-05\\
404.105002	1.14278809982691e-05\\
503.437239	7.24631706941632e-06\\
};
\addlegendentry{ETD2RK \textsc{phiks}}

\addplot [color=red, line width=1.5pt, dashed, mark size=3pt, mark=o, mark options={solid, red}]
  table[row sep=crcr]{%
33.918777	0.000171947836794551\\
67.961675	3.53117849227183e-05\\
94.430805	1.4750630247508e-05\\
126.129618	8.04729551435252e-06\\
163.329652	5.05633673589586e-06\\
};
\addlegendentry{exponential Rosenbrock--Euler \textsc{phisplit}}

\addplot [color=blue, line width=1.5pt, dashed, mark size=3pt, mark=diamond, mark options={solid, blue}]
  table[row sep=crcr]{%
15.80114	0.000197437078730741\\
32.79077	4.04426845570676e-05\\
48.695702	1.6879679313532e-05\\
65.422701	9.2051782651072e-06\\
80.723248	5.7826719555826e-06\\
};
\addlegendentry{ETD2RK \textsc{phisplit}}

\addplot [color=magenta, line width=1.5pt, dashed, mark size=3pt, mark=asterisk, mark options={solid, magenta}]
  table[row sep=crcr]{%
127.078858	0.000197588464282935\\
127.629446	5.0930548888984e-05\\
127.786372	1.63290209233416e-05\\
128.472231	9.83416364349182e-06\\
129.517735	5.49795058708216e-06\\
};
\addlegendentry{\texttt{ode23}}

\end{axis}

\end{tikzpicture}%
  \caption{Achieved errors in the Frobenius norm
    and wall-clock times in seconds for the solution
    of Riccati differential equation~\eqref{eq:exric} up to final
    time $T=0.025$, with different
    integrators and $\hat{n}=40$. The number of time steps for each exponential
    method is reported in
    Table~\ref{tab:riccati_order_40}. The input tolerances (both absolute
    and relative) for \texttt{ode23} are $3\mathrm{e}{-3}$,
    $4\mathrm{e}{-4}$, $2.3\mathrm{e}{-4}$, $9.5\mathrm{e}{-5}$,
    and $8.7\mathrm{e}{-5}$.}
  \label{fig:riccati_40}
\end{figure}

\begin{remark}
  The discretization of the operator~\eqref{eq:advdiff2dric}
  has itself a Kronecker sum
structure. Hence, it is possible to write equation~\eqref{eq:exric}
(in vector formulation for simplicity of exposition) as
\begin{equation*}
\left\{
\begin{aligned}
    \boldsymbol u'(t) &= K \boldsymbol u + \boldsymbol g(\boldsymbol u),\\
    \boldsymbol u(0) &= \boldsymbol z,
\end{aligned}
\right.
\end{equation*}
where $\boldsymbol g$ and $\boldsymbol z $ are the vectorization
of the nonlinearity and of $\boldsymbol Z$, respectively, and $K$ has
the form
\begin{equation*}
  K = I \otimes I \otimes I \otimes D_1^{\mathsf{T}} +
      D_2^{\mathsf{T}} \otimes I \otimes I \otimes I.
\end{equation*}
Here $I$ is an identity matrix of size $\hat{n}\times\hat{n}$ and
$D_1\in\RR^{\hat{n}\times\hat{n}}$ and $D_2\in\RR^{\hat{n}\times\hat{n}}$
the discretizations of the operators $\partial_{xx}-10x\partial_x$ and
$\partial_{yy}-100y\partial_y$, respectively. In the context of temporal
integration with exponential Runge--Kutta schemes, we could then use
both the \textsc{phiks} and the \textsc{phisplit} approaches with the
even smaller sized matrices $D_1$ and $D_2$, forming then the approximations at
every time steps using Tucker operators with
order-4 tensors. However, as
this is just possible because of the specific form of the
operator~\eqref{eq:advdiff2dric}, we do not pursue this approach here.
\end{remark}

\subsection{Advection--diffusion--reaction}
We now consider the semidiscretization in space
of the following three-dimensional evolutionary
Advec\-tion--Diffusion--Reaction (ADR) equation (see Reference~\cite{CCZ23})
  \begin{equation}\label{eq:ADR}
    \left\{\begin{aligned}
  \partial_t u(t,x_1,x_2,x_3)&=\varepsilon \Delta u(t,x_1,x_2,x_3)+
  \alpha(\partial_{x_1} +\partial_{x_2}+\partial_{x_3})u(t,x_1,x_2,x_3)\\
  &+g(t,x_1,x_2,x_3,u(t,x_1,x_2,x_3)),\\
  u_0(x_1,x_2,x_3) &= 64x_1(1-x_1)x_2(1-x_2)x_3(1-x_3).
\end{aligned}\right.
  \end{equation}
The nonlinear function $g$ is given by
\begin{equation*}
  g(t,x_1,x_2,x_3,u(t,x_1,x_2,x_3))=\frac{1}{1+u(t,x_1,x_2,x_3)^2}+
  \Psi(t,x_1,x_2,x_3),
\end{equation*}
where $\Psi(t,x_1,x_2,x_3)$ is such that the analytical
solution of the equation is
\begin{equation*}
  u(t,x_1,x_2,x_3)=\ee^t u_0(x_1,x_2,x_3).
\end{equation*}
The problem is solved up to final time $T=1$ in the domain $[0,1]^3$ and
completed with homogeneous Dirichlet
boundary conditions.
The remaining parameters are set to $\varepsilon=0.75$ and $\alpha=0.1$.
By semidiscretizing in space with second order centered
finite differences, we obtain
a system of type~\eqref{eq:ODE} with $K$
having in three-dimensional Kronecker sum
structure,
where $A_\mu$ approximates
$\varepsilon\partial_{x_\mu x_\mu}+\alpha\partial_{x_\mu}$.
We first perform simulations with
$n_1=40$, $n_2=41$, and $n_3=42$ inner discretization
points for the $x_1$, $x_2$ and $x_3$ variables, respectively.
The temporal integration is performed with four methods:
the Lawson--Euler scheme~\eqref{eq:lawsoneuler}, the exponential Euler
method~\eqref{eq:expEuler}, the Lawson2b scheme~\eqref{eq:lawson2}
and the ETD2RK method~\eqref{eq:ETD2RK}
(see Section~\ref{sec:methods} for their
practical implementation and the \textsc{phisplit} versions).
In particular, concerning the Lawson schemes, the needed matrix exponentials
$\exp(\tau A_\mu)$, with $\mu=1,2,3$,
are computed once and for all at the beginning. Then,
one and two Tucker operators per time step, for the
first order and second order scheme, respectively, are required
to form the approximations during the temporal integration.
Concerning the \textsc{phisplit} implementation of exponential
Euler and ETD2RK, again we compute once and for all the needed
matrix functions $\varphi_1(\tau A_\mu)$
and $\varphi_2(\tau A_\mu)$ before starting the temporal
integration, and we then combine them suitably at each time step. This
operation requires
a single Tucker operator for the first order scheme and two for the second
order one, as for the aforementioned Lawson schemes, plus the action $K\bb u_n$.

The number of time steps for each
method, for both the \textsc{phisplit} and \textsc{phiks} implementations, is
reported in Table~\ref{tab:adr_order_40}, while the reached relative errors and
the wall-clock times are summarized in Figure~\ref{fig:adr_ord12_40}.
\begin{table}[!htb]
  \centering
  \bgroup
  \def\arraystretch{1.2}
  \begin{tabular}{c||c|ccccc}
  \multirow{2}{*}{Lawson--Euler} &
  steps & 800 & 8800 & 16800 & 24800 & 32800 \\\cline{2-7}
  &order &  --  & 1.00 & 1.00 & 1.00 & 1.00 \\
  \hline
  \multirow{2}{*}{exponential Euler \textsc{phiks}} &
  steps & 50 & 450 & 850 & 1250 & 1650 \\\cline{2-7}
  &order & --   & 1.03 & 1.00 & 1.00 & 1.00 \\
  \hline
  \multirow{2}{*}{exponential Euler \textsc{phisplit}} &
  steps & 50 & 450 & 850 & 1250 & 1650 \\\cline{2-7}
  & order &  --  & 1.03 & 1.01 & 1.00 & 1.00 \\
  \multicolumn{7}{c}{}\\
  \multirow{2}{*}{Lawson2b} &
  steps & 1500 & 5500 & 9500 & 13500 & 17500 \\\cline{2-7}
  &order &  --  & 1.96 & 1.99 & 2.00 & 2.00 \\
  \hline
  \multirow{2}{*}{ETD2RK \textsc{phiks}} &
  steps & 20 & 80 & 140 & 200 & 260 \\\cline{2-7}
  &order &  --  & 1.94 & 1.97 & 1.98 & 1.99 \\
  \hline
  \multirow{2}{*}{ETD2RK \textsc{phisplit}} &
  steps & 40 & 140 & 240 & 340 & 440 \\\cline{2-7}
  & order & --   & 2.10 & 2.04 & 2.03 & 2.02 \\
  \end{tabular}
  \egroup
  \caption{Number of time steps and observed convergence rates for the
    time integration of
    the semidiscretization of the ADR
    equation~\eqref{eq:ADR} up to
    final time $T=1$ , with different exponential integrators.
    Here we considered $n_1=40$, $n_2=41$ and $n_3=42$ inner space
    discretization
    points for the $x_1$, $x_2$ and $x_3$ variables, respectively.
    The achieved errors and the wall-clock times are displayed in
    Figure~\ref{fig:adr_ord12_40}.}
  \label{tab:adr_order_40}
\end{table}%
\begin{figure}[!htb]
  \centering
%
%
%
\begin{tikzpicture}

\begin{axis}[%
width=3.8in,
height=2.5in,
at={(0.769in,0.477in)},
scale only axis,
xmin=0,
xmax=22,
xlabel style={font=\color{white!15!black}},
xlabel={wall-clock time},
ymode=log,
ymin=1e-4,
ymax=1e-1,
yminorticks=true,
ylabel style={font=\color{white!15!black}},
ylabel={Relative error},
axis background/.style={fill=white},
legend style={legend cell align=left, align=left, draw=white!15!black, font=\scriptsize}
]

\addplot [color=teal,line width=1.5pt, mark size=3pt, dotted, mark=x, mark options={solid, teal}]
  table[row sep=crcr]{%
0.222823	0.0117169705391758\\
2.669643	0.00106791877453733\\
4.922048	0.000559454399926992\\
7.768094	0.000379001674524301\\
10.006463	0.000286568607256026\\
};
\addlegendentry{Lawson--Euler}

\addplot [color=orange,line width=1.5pt, mark size=3pt, mark=triangle, mark options={solid, orange}]
  table[row sep=crcr]{%
0.438731000000001	0.0106292111502372\\
2.723039	0.00110753252968165\\
4.718193	0.000584623666850009\\
6.512938	0.000397123417112195\\
9.113223	0.000300686845803868\\
};
\addlegendentry{exponential Euler \textsc{phiks}}

\addplot [color=orange,line width=1.5pt, mark size=3pt, dashed, mark=triangle, mark options={solid, orange}]
  table[row sep=crcr]{%
0.0992120000000001	0.0107081236857995\\
0.384999	0.00111027419393683\\
0.660487	0.000585386452197191\\
0.871026	0.000397475190801764\\
1.020225	0.000300888458785554\\
};
\addlegendentry{exponential Euler \textsc{phisplit}}

\end{axis}

\end{tikzpicture}%
%
%
%
\begin{tikzpicture}

\begin{axis}[%
width=3.8in,
height=2.5in,
at={(0.758in,0.481in)},
scale only axis,
xmin=0,
xmax=22,
xlabel style={font=\color{white!15!black}},
xlabel={wall-clock time},
ymode=log,
ymin=1e-06,
ymax=0.001,
yminorticks=true,
ylabel style={font=\color{white!15!black}},
ylabel={Relative error},
axis background/.style={fill=white},
legend style={legend cell align=left, align=left, draw=white!15!black, font=\scriptsize}
]

\addplot [color=teal,line width=1.5pt, mark size=3pt, dotted, mark=square, mark options={solid, teal}]
  table[row sep=crcr]{%
0.945534	0.000315674843375309\\
3.631012	2.487068517044e-05\\
6.047837	8.3649292181883e-06\\
8.524288	4.14595608695139e-06\\
11.644669	2.46812995166966e-06\\
};
\addlegendentry{Lawson2b}

\addplot [color=blue,line width=1.5pt, mark size=3pt, mark=diamond, mark options={solid, blue}]
  table[row sep=crcr]{%
0.367943	0.000307538535906832\\
1.362085	2.09342120457308e-05\\
2.348924	6.94077686699274e-06\\
3.287802	3.42273926898397e-06\\
3.791115	2.03237304479068e-06\\
};
\addlegendentry{ETD2RK \textsc{phiks}}

\addplot [color=blue,line width=1.5pt, mark size=3pt, dashed, mark=diamond, mark options={solid, blue}]
  table[row sep=crcr]{%
0.0731409999999997	0.0003034138277914\\
0.163673	2.19309513095735e-05\\
0.278888	7.30092898625085e-06\\
0.340922	3.60467279794151e-06\\
0.450544	2.14158850419228e-06\\
};
\addlegendentry{ETD2RK \textsc{phisplit}}

\addplot [color=magenta,line width=1.5pt, mark size=3pt, dashed, mark=asterisk, mark options={solid, magenta}]
  table[row sep=crcr]{%
14.49675	1.6976e-05\\
19.624603	5.6487e-06\\
20.358253	2.8646e-06\\
21.577953	2.1343e-06\\
};
\addlegendentry{\texttt{ode23t}}

\end{axis}

\end{tikzpicture}%
  \caption{Achieved errors in the infinity norm
    and wall-clock times in seconds for the solution
    of the semidiscretization of the ADR
    equation~\eqref{eq:ADR} up to final time $T=1$, with different exponential
    integrators of order one (top) and order two (bottom).
    Here we considered $n_1=40$, $n_2=41$ and $n_3=42$ inner space
    discretization
    points for the $x_1$, $x_2$ and $x_3$ variables, respectively.
    The number of time steps for each exponential
    method is reported in
    Table~\ref{tab:adr_order_40}. The input tolerances (both absolute and
    relative) for \texttt{ode23t} are $8\mathrm{e}{-3}$,
    $4\mathrm{e}{-5}$, $1\mathrm{e}{-5}$, and $5\mathrm{e}{-6}$.}
  \label{fig:adr_ord12_40}
\end{figure}%
First of all, we notice that all the methods show the expected convergence rate,
reported in Table~\ref{tab:adr_order_40} as well.
The Lawson--Euler method and the exponential Euler scheme
in its \textsc{phiks} implementation (see
top plot of Figure~\ref{fig:adr_ord12_40}) perform equally well,
even if the former requires much more time steps. Overall, the exponential
Euler method in its \textsc{phisplit} variant is roughly 10 times faster to
reach the highest accuracy in this experiment.
If we consider the second order methods
(bottom plot of Figure~\ref{fig:adr_ord12_40}), we observe that the Lawson2b
scheme needs much more wall-clock time to reach the same level of accuracy of
the other methods, and overall the best performant method is ETD2RK in the
\textsc{phisplit} variant.
In this plot we report also the results obtained with the internal
\textsc{matlab}
\texttt{ode23t} integrator.
It is an implicit Runge--Kutta method of order three with variable
step size, which
is suggested for stiff problems at low accuracies. Nevertheless, it
performs worse than the considered exponential integrators.

Finally, we repeat the experiment with
$n_1=80$, $n_2=81$, and $n_3=82$ inner discretization
points for the $x_1$, $x_2$ and $x_3$ variables, respectively.
The number of time steps for each
method is
reported in Table~\ref{tab:adr_order}, while the  relative errors and
the wall-clock times are summarized in Figure~\ref{fig:adr_ord12}.
\begin{table}[!htb]
  \centering
  \bgroup
  \def\arraystretch{1.2}
  \begin{tabular}{c||c|ccccc}
  \multirow{2}{*}{Lawson--Euler} &
  steps & 800 & 8800 & 16800 & 24800 & 32800 \\\cline{2-7}
  &order &  --  & 1.00 & 1.00 & 1.00 & 1.00 \\
  \hline
  \multirow{2}{*}{exponential Euler \textsc{phiks}} &
  steps & 50 & 450 & 850 & 1250 & 1650 \\\cline{2-7}
  &order & --   & 1.02 & 1.00 & 1.00 & 1.00 \\
  \hline
  \multirow{2}{*}{exponential Euler \textsc{phisplit}} &
  steps & 50 & 450 & 850 & 1250 & 1650 \\\cline{2-7}
  & order &  --  & 1.03 & 1.01 & 1.00 & 1.00 \\
  \multicolumn{7}{c}{}\\
  \multirow{2}{*}{Lawson2b} &
  steps & 3000 & 4500 & 6000 & 7500 & 9000 \\\cline{2-7}
  &order &  --  & 1.79 & 1.87 & 1.92 & 1.94 \\
  \hline
  \multirow{2}{*}{ETD2RK \textsc{phiks}} &
  steps & 20 & 80 & 140 & 200 & 260 \\\cline{2-7}
  &order &  --  & 1.94 & 1.97 & 1.98 & 1.99 \\
  \hline
  \multirow{2}{*}{ETD2RK \textsc{phisplit}} &
  steps & 40 & 140 & 240 & 340 & 440 \\\cline{2-7}
  & order & --   & 2.10 & 2.04 & 2.03 & 2.02 \\
  \end{tabular}
  \egroup
  \caption{Number of time steps and observed convergence rates for the
    time integration of
    the semidiscretization of the ADR
    equation~\eqref{eq:ADR} up to
    final time $T=1$ , with different exponential integrators.
    Here we considered $n_1=80$, $n_2=81$ and $n_3=82$ inner space
    discretization
    points for the $x_1$, $x_2$ and $x_3$ variables, respectively.
    The achieved errors and the wall-clock times are displayed in
    Figure~\ref{fig:adr_ord12}.}
  \label{tab:adr_order}
\end{table}%
\begin{figure}[!htb]
  \centering
%
%
\begin{tikzpicture}

\begin{axis}[%
width=3.8in,
height=2.5in,
at={(0.758in,0.481in)},
scale only axis,
xmin=0,
xmax=160,
xlabel style={font=\color{white!15!black}},
xlabel={wall-clock time},
ymode=log,
ymin=1e-4,
ymax=1e-1,
yminorticks=true,
ylabel style={font=\color{white!15!black}},
ylabel={Relative error},
axis background/.style={fill=white},
legend style={legend cell align=left, align=left, draw=white!15!black,font=\scriptsize},
]

\addplot [color=teal,line width=1.5pt, dotted, mark size=3pt, mark=x, mark options={solid, teal}]
  table[row sep=crcr]{%
3.830015	0.0117093166959706\\
40.023919	0.00106722437807435\\
77.239923	0.000559090769518005\\
112.165384	0.000378755369010267\\
149.113743	0.000286382386118996\\
};
\addlegendentry{Lawson--Euler}

\addplot [color=orange,line width=1.5pt, mark size=3pt, mark=triangle, mark options={solid, orange}]
  table[row sep=crcr]{%
2.229278	0.0104891960292066\\
17.652909	0.00110752656870146\\
33.506459	0.000584621978358339\\
46.286751	0.000397122627788412\\
60.386908	0.000300686387764159\\
};
\addlegendentry{exponential Euler \textsc{phiks}}

\addplot [color=orange,line width=1.5pt, dashed, mark size=3pt, mark=triangle, mark options={solid, orange}]
  table[row sep=crcr]{%
0.603194999999999	0.0107075742367188\\
4.960021	0.00111026771878275\\
9.151983	0.000585384629859972\\
13.933355	0.000397474341289116\\
17.860089	0.000300887966703942\\
};
\addlegendentry{exponential Euler \textsc{phisplit}}

\end{axis}

\end{tikzpicture}%
%
%
\begin{tikzpicture}

\begin{axis}[%
width=3.8in,
height=2.5in,
at={(0.769in,0.477in)},
scale only axis,
xmin=0,
xmax=160,
xlabel style={font=\color{white!15!black}},
xlabel={wall-clock time},
ymode=log,
ymin=1e-06,
ymax=0.001,
yminorticks=true,
ylabel style={font=\color{white!15!black}},
ylabel={Relative error},
axis background/.style={fill=white},
legend style={legend cell align=left, align=left, draw=white!15!black,font=\scriptsize},
]
\addplot [color=teal,line width=1.5pt, dotted, mark size=3pt, mark=square, mark options={solid, teal}]
  table[row sep=crcr]{%
33.380836	0.000258374265688116\\
51.170973	0.000124894648381687\\
67.920979	7.28638588362354e-05\\
87.562451	4.75123922123985e-05\\
102.31712	3.33502762327595e-05\\
};
\addlegendentry{Lawson2b}

\addplot [color=blue,line width=1.5pt, mark size=3pt, mark=diamond, mark options={solid, blue}]
  table[row sep=crcr]{%
2.420618	0.000307505689306091\\
9.45234	2.0934011403357e-05\\
16.313003	6.94075073692446e-06\\
22.767557	3.42273413204819e-06\\
29.234035	2.03236901281894e-06\\
};
\addlegendentry{ETD2RK \textsc{phiks}}

\addplot [color=blue,line width=1.5pt, dashed, mark size=3pt, mark=diamond, mark options={solid, blue}]
  table[row sep=crcr]{%
0.733109	0.000303319250943859\\
2.541384	2.19248936111415e-05\\
4.396807	7.29905280522352e-06\\
5.868383	3.60379248212235e-06\\
7.815809	2.14108285352612e-06\\
};
\addlegendentry{ETD2RK \textsc{phisplit}}

\end{axis}
\end{tikzpicture}%
  \caption{Achieved errors in the infinity norm
    and wall-clock times in seconds for the solution
    of the semidiscretization of the ADR
    equation~\eqref{eq:ADR} up to final time $T=1$, with different exponential
    integrators of order one (top) and order two (bottom).
    Here we considered $n_1=80$, $n_2=81$ and $n_3=82$ inner space
    discretization
    points for the $x_1$, $x_2$ and $x_3$ variables, respectively.
    The number of time steps for each
    method is reported in
    Table~\ref{tab:adr_order}.}
  \label{fig:adr_ord12}
\end{figure}%
Again, we notice that all the methods show the expected convergence rate,
reported in Table~\ref{tab:adr_order} as well.
In particular, for large time step sizes, the Lawson2b method
suffers from an order reduction.
This is expected, as in these cases the problem
is more stiff, and schemes which employ just the exponential function are
affected by this phenomenon (see, for instance, Reference~\cite{HLO20}).
Then, from Figure~\ref{fig:adr_ord12}
we observe that the \textsc{phisplit} approach
is in any case the most performant among all the methods and techniques
considered,
with an increasing speedup for more stringent accuracies. More in detail,
compared with its \textsc{phiks} counterparts,
the \textsc{phisplit} implementations are roughly
3.5 time faster,
even if (in general) they require more time steps to reach a
comparable accuracy.
Finally, the Lawson schemes are always the least performant.
This is mainly due to the requirement of a large number of time steps to
reach the accuracy of the other methods, which is
particularly evident for the second order
schemes (see bottom of Table~\ref{tab:adr_order} and
Figure~\ref{fig:adr_ord12}). Moreover, while ETD2RK in its
\textsc{phisplit} variant reached the most stringent
accuracy in less than 10 seconds, Lawson2b was not able
to reach an accuracy 10 times larger in 100 seconds. Hence,
we decided to stop the simulation with this integrator
with a larger number of time steps.
Finally, concerning the internal \textsc{matlab} ODE suite, none of the methods
was able to output a solution within 10 minutes.

\section{Conclusions}\label{sec:conclusions}
In this paper, we presented how it is possible to efficiently approximate
actions of $\varphi$-functions for matrices with $d$-dimensional Kronecker sum
structure using a $\mu$-mode based approach. The technique, that we call
\textsc{phisplit}, is suitable when integrating initial valued ordinary
differential equations with exponential integrators up to second order.
It is based on an inexact direction splitting of the matrix functions
involved in the
time marching schemes which preserves the order of the method. The effectiveness
and superiority of the approach, with respect to another technique to compute
actions of $\varphi$-functions in Kronecker form, has been shown on a
two-dimensional problem from linear quadratic control and on a three-dimensional
advection--diffusion--reaction equation, using a variety of exponential
integrators. Interesting future developments would be to generalize the
approach for higher order integrators and performing GPU simulations with the
\textsc{phisplit} technique, possibly in single and/or half precision
(which are compatible with the magnitude of the errors of the temporal
integration), for different problems from science and engineering.
\section*{Acknowledgments}
The authors acknowledge partial support from the Program Ricerca di Base
2019 No.~RBVR199YFL
of the University of Verona entitled ``Geometric Evolution of Multi Agent
Systems''.

\bibliographystyle{plain}
    \bibliography{phisplit}
\end{document}